\documentclass[11pt]{article}

\usepackage[margin=1in]{geometry}
\usepackage{amsmath,amssymb,amsthm,mathtools}
\usepackage{authblk}
\usepackage{microtype}
\usepackage[hidelinks]{hyperref}

\newtheorem{theorem}{Theorem}[section]
\newtheorem{lemma}[theorem]{Lemma}

\newcommand{\cL}{c_{\mathrm{L}}}
\newcommand{\Lev}{\mathcal{L}}

\title{Improved bounds for the lazy cops and robbers \\
on generalized hypercubes}
\author[]{Anand Babu}
\author[]{Ashwin Jacob}
\author[]{Karunakaran Murali Krishnan}
\author[]{Reshma Roy}
\author[]{Sreekala S}
\affil[]{Department of Computer Science \& Engineering,\\
National Institute of Technology Calicut, Kozhikode, India}
\date{}

\begin{document}
\maketitle

\begin{abstract}
In Lazy Cops and Robbers, at most one cop moves on each cop turn.  We study
the lazy cop number of the generalized hypercube $Q(n,m)$, whose vertex set
is ${\{0,1,\ldots,m\}}^n$.  For each fixed integer $m\geq2$, we prove the
asymptotic upper bound
$$
 \cL(Q(n,m))
 =O\!\left(\frac{{(m+1)}^n}{n^{3/2}}\right).
$$
This result improves the upper bound of Sim, Tan, and Wong by a factor of
$\log n$.  The proof combines a moving dominating-set argument with an explicit
dominating-set construction inside the support classes of each level. 
As a separate domination result, we show that, for fixed integers $m\geq2$
and $d\geq1$, the Hamming graph $K_m^{\square k}$ has a distance-$d$
dominating set of asymptotic size $O(m^k/k^d)$. This order is
optimal up to a constant factor.
\end{abstract}

\section{Introduction}

Let $G$ be a finite connected graph.  In \emph{Lazy Cops and Robbers}, the
cops first choose their initial vertices, after which the robber chooses an
initial vertex.  The cops and the robber then alternate turns, beginning
with the cops.  On a cop turn, either all cops remain stationary or exactly
one cop moves along an edge.  On a robber turn, the robber either remains
stationary or moves along an edge.  Capture occurs when a cop occupies the
robber's vertex.  The \emph{lazy cop number} $\cL(G)$ is the minimum number
of cops that can force capture after finitely many turns.

For each positive integer $n$, let $[n]=\{1,\ldots,n\}$.  For positive
integers $n$ and $m$, the \emph{generalized hypercube} $Q(n,m)$ has vertex
set ${\{0,1,\ldots,m\}}^n$.  Two vertices are adjacent when they differ in
exactly one coordinate.

Bal, Bonato, Kinnersley, and Pra\l at~\cite{BalEtAl} studied Lazy Cops and
Robbers on hypercubes.  Offner and Ojakian~\cite{OffnerOjakian} developed
a strategy that successively guards the levels of a hypercube, and Sim,
Tan, and Wong~\cite{SimTanWong} extended the method to generalized
hypercubes.  For
fixed $m\geq2$, Sim, Tan, and Wong proved the asymptotic upper
bound~\cite[Theorem~2.5]{SimTanWong}
\[
 \cL(Q(n,m))
 =O\!\left(\frac{{(m+1)}^n\log n}{n^{3/2}}\right).
\]
Their logarithmic factor arises from applying a general upper bound for the
domination number in terms of the order and minimum degree to the graphs
induced by the levels; see~\cite[Theorem~1.2.2]{AlonSpencer}.  We replace
that general estimate with an explicit algebraic construction.

Our main result is the following.

\begin{theorem}\label{thm:main}
For each fixed integer $m\geq2$ and for sufficiently large $n$, we have
\[
 \cL(Q(n,m))
 =O\!\left(\frac{{(m+1)}^n}{n^{3/2}}\right).
\]
\end{theorem}

The algebraic construction admits a natural extension to distance
domination in Hamming graphs.  We state and prove this extension separately
in Section~\ref{sec:distance-domination}.  The case $d=1$ supplies the
dominating sets used to prove Theorem~\ref{thm:main}.

\section{Distance Domination in Hamming Graphs}\label{sec:distance-domination}

The Hamming graph $K_m^{\square k}$ has words of length $k$ over the alphabet $[m]$ as its vertices, and
two vertices are adjacent when they differ in exactly one coordinate. 
%The Hamming graph $K_m^{\square k}$ has vertex set ${[m]}^k$.  Two vertices are adjacent if and only if they differ in exactly one coordinate.  

A set
$C$ is a \emph{distance-$d$ dominating set} if every vertex of the graph is
at distance at most $d$ from a vertex of $C$.

\begin{theorem}\label{thm:hamming-domination}
For integers $m\geq2$ and $k,d\geq1$, the graph $K_m^{\square k}$ has a
distance-$d$ dominating set $C$ satisfying
\[
 |C|=m^{k-d\lfloor\log_m(k/d+1)\rfloor}
 <\frac{d^d m^{k+d}}{{(k+d)}^d}.
\]
\end{theorem}

\begin{proof}
Let $A=\mathbb Z_m=\{0_A,1_A,\ldots,{(m-1)}_A\}$ be the additive cyclic
group of order $m$.  Choose a bijection between $[m]$ and $A$, and use it
to identify the vertices of $K_m^{\square k}$ with the elements of $A^k$.

Set
$t=\left\lfloor\log_m(k/d+1)\right\rfloor$.
Raising to the power of $m$ in both sides, we get $m^t\leq k/d+1$, and thus,  $d(m^t-1)\leq k$. Put $p=m^t-1$. 

Let
$H=\bigoplus_{\ell=1}^d A^t $ 
be the direct sum of $d$ copies of $A^t$ (thus, an element of $H$ contains $d$ components, each belonging to $A^t$).  Enumerate
the elements of $A^t$ as 
\[
 A^t=\{v_0,v_1,\ldots,v_p\}, \textrm{ where } v_0=0_{A^t}.
 \]
For $1\leq\ell\leq d$ and $1\leq r\leq p$, define 
\[
 a_{(\ell-1)p+r}
 =(0_{A^t},\ldots,0_{A^t},v_r,
   0_{A^t},\ldots,0_{A^t})\in H,
\]
where $v_r$ occurs in the $\ell$th summand.  For $dp<j\leq k$, set
$a_j=0_H$.

Represent each element of $A$ by its unique integer in
$\{0,1,\ldots,m-1\}$.  For $x=(x_1,\ldots,x_k)\in A^k$, define a map $\phi:A^k\to H$ as
\[
 \phi(x)=\sum_{j=1}^k x_j a_j,
\]
where $x_j a_j$ denotes the sum of $x_j$ copies of $a_j$ in $H$.  Since $H$ is built from copies of  $\mathbb{Z}_m$, adding any element $a_j$ in $H$ $m$ times gives the identity element $0_H \in H$. Therefore, reducing a coefficient modulo $m$ does not change the sum. Thus, $\phi$ is a group
homomorphism.  Let $C=\ker\phi$.

If $t=0$, then $H$ is trivial, so $\phi$ is surjective. Suppose that
$t\geq1$.  For each $\ell$, the labels indexed from $(\ell-1)p+1$
through $\ell p$ include the vectors having a standard basis vector of
$A^t$ in the $\ell$th summand and zero vectors in all other summands.
These labels generate $H$, so $\phi$ is again surjective.  The First
Isomorphism Theorem now gives 
\[
 |C|=\frac{|A^k|}{|H|}= \frac{m^k}{m^{dt}} = m^{k-dt}.
\]
Fix $x\in A^k$, and write $
 \phi(x)=(\sigma_1,\ldots,\sigma_d)$,
where $\sigma_\ell\in A^t$.  For each $\ell$ such that
$\sigma_\ell\ne0_{A^t}$, let $r_\ell$ be the unique index in
$\{1,\ldots,p\}$ satisfying $\sigma_\ell=v_{r_\ell}$, and put $
 j_\ell=(\ell-1)p+r_\ell$.
The indices $j_\ell$ are distinct.  
Let $e_j$ denote the element of $A^k$ having $1_A$ in coordinate $j$
and $0_A$ elsewhere (like the standard basis vector). Set
\[
 x'=x-\sum_{\ell:\,\sigma_\ell\ne0_{A^t}}e_{j_\ell}.
\]
The elements $x$ and $x'$ differ in at most $d$ coordinates.  Furthermore,
\[
 \phi(x')
 =\phi(x)-\sum_{\ell:\,\sigma_\ell\ne0_{A^t}}
   \phi(e_{j_\ell})
 =\phi(x)-\sum_{\ell:\,\sigma_\ell\ne0_{A^t}}a_{j_\ell}
 =0_H.
\]

Thus $x'\in C$, and every vertex of $K_m^{\square k}$ is at distance at
most $d$ from $C$.

Finally, since $t> \log_m(k/d+1) - 1$,  we have $
 m^t>\frac{k/d+1}{m}=\frac{k+d}{dm}$.
Therefore, 
\[
|C|=m^{k-dt}
 <m^k{\left(\frac{dm}{k+d}\right)}^d
 =\frac{d^d m^{k+d}}{{(k+d)}^d}.
 \]
\end{proof}

A radius-$d$ ball in $K_m^{\square k}$ has
\[
 V_{m,k}(d)
 =\sum_{i=0}^{\min\{d,k\}}\binom{k}{i}{(m-1)}^i
\]
vertices.  If $D$ is a distance-$d$ dominating set, then these balls,
centered at the vertices of $D$, cover all $m^k$ vertices.  Consequently,
$ |D|V_{m,k}(d)\geq m^k$, and hence
\[
 |D|\geq\frac{m^k}{V_{m,k}(d)}.
\]
For fixed $m$ and $d$, we have $V_{m,k}(d)=\Theta(k^d)$.
Thus the construction in Theorem~\ref{thm:hamming-domination} has order
$\Theta(m^k/k^d)$ and is asymptotically optimal up to a constant factor.

\section{The Lazy Cop Number of Generalized Hypercubes}

\subsection{The Strategy}

For $x=(x_1,\ldots,x_n)\in V(Q(n,m))$, define the \emph{support} of $x$
by $\operatorname{supp}(x)=\{j\in[n]\colon x_j\ne0\}$. For $k\in\{0,\ldots,n\}$, the \emph{$k$th level} is $
 \Lev_k=\{x\in V(Q(n,m))\colon |\operatorname{supp}(x)|=k\}$.
A set $B$ \emph{dominates} a vertex set $U$ if every vertex of $U$ either
belongs to $B$ or has a neighbor in $B$.

The following lemma formalizes the method of Offner and
Ojakian~\cite{OffnerOjakian}, as extended to generalized hypercubes by Sim,
Tan, and Wong~\cite{SimTanWong}.

\begin{lemma}\label{lem:moving-sets}
Suppose that, for every $k\in\{0,\ldots,n\}$, a set $B_k$ of at most
$M$ vertices dominates $\Lev_k$.  Then
$ \cL(Q(n,m))\leq2M$.
\end{lemma}

\begin{proof}
Place two teams of $M$ cops so that each team occupies every vertex of
$B_n$.  If $|B_n|<M$, place the unused cops at any occupied vertex.  If the
robber starts in $\Lev_n$, then a cop captures the robber on the first cop
turn.  Hence we may assume that the robber starts below $\Lev_n$.

Suppose that one team occupies $B_k$ and that the robber is below
$\Lev_k$.  Keep this team on $B_k$ while the cops in the other team move,
one at a time, until they occupy every vertex of $B_{k-1}$.  Such a
transfer is possible because $Q(n,m)$ is connected.  Every edge of
$Q(n,m)$ changes the support size by at most one.  Therefore, during the
transfer, the robber cannot move from below $\Lev_k$ to above it without
entering $\Lev_k$.  If the robber enters $\Lev_k$, then a cop on $B_k$
captures the robber on the next cop turn.

Once the second team occupies $B_{k-1}$, the robber can survive the next
cop turn only by moving below $\Lev_{k-1}$.  Indeed, a robber that remains
in $\Lev_{k-1}$ is captured from $B_{k-1}$, while a robber that moves into
$\Lev_k$ is captured from $B_k$.  The teams may therefore exchange roles
and repeat the transfer with $k-1$ in place of $k$.

Continuing downward eventually places a team on $B_0$.  At that point the
robber cannot move to a lower level, so capture follows.
\end{proof}

\subsection{Proof of the Main Theorem}

Fix $k\in[n]$ and a set $S\subseteq[n]$ with $|S|=k$.  Let $
 V_S=\{x\in\Lev_k\colon \operatorname{supp}(x)=S\}$. Each coordinate indexed by $S$ can take one of $m$ nonzero values.
Consequently, $
 Q(n,m)[V_S]\cong K_m^{\square k}$.
The sets $V_S$, over all $k$-element subsets $S$ of $[n]$, partition
$\Lev_k$ into $\binom nk$ support classes.

Apply Theorem~\ref{thm:hamming-domination} with $d=1$ independently to
each support class.  Let $B_k$ be the union of the resulting dominating
sets, and set $B_0=\Lev_0$.  Then $B_k$ dominates $\Lev_k$, and, for
every $k\in\{0,\ldots,n\}$,
\begin{equation}\label{eq:Bk}
 |B_k|
 <\binom nk\frac{m^{k+1}}{k+1}.
\end{equation}

Put $q=m+1$ and $\alpha=m/(m+1)$, and let
$X\sim\operatorname{Bin}(n,\alpha)$.  A direct calculation gives
\[
 \Pr(X=k)
 =\binom nk\alpha^k{(1-\alpha)}^{n-k}
 =\binom nk\frac{m^k}{q^n}.
\]
Hence,
\begin{equation}\label{eq:binomial-identity}
 \binom nk m^k=q^n\Pr(X=k).
\end{equation}

By the unimodality of the binomial distribution and Stirling's
formula~\cite[Section~3.6.2]{Cameron}, there is a constant $D_m>0$ such
that, for all sufficiently large $n$ and every
$k\in\{0,\ldots,n\}$,
$ \Pr(X=k)\leq D_m/\sqrt n$.
If $k\geq\alpha n/2$, then equations~\eqref{eq:Bk}
and~\eqref{eq:binomial-identity} give
\[
 |B_k|
 <\frac{m}{k+1}q^n\Pr(X=k)
 \leq\frac{2mD_m}{\alpha}\frac{q^n}{n^{3/2}}.
\]

Now suppose that $k<\alpha n/2$.  The multiplicative Chernoff lower-tail
inequality~\cite[Appendix~A]{AlonSpencer} gives
\[
 \Pr\!\left(X\leq\frac{\alpha n}{2}\right)
 \leq\exp\!\left(-\frac{\alpha n}{8}\right).
\]
Since $B_k\subseteq\Lev_k$, we have
\[
 |B_k|
 \leq|\Lev_k|
 =q^n\Pr(X=k)
 \leq q^n\exp\!\left(-\frac{\alpha n}{8}\right)
 \leq\frac{q^n}{n^{3/2}}
\]
for all sufficiently large $n$. Both ranges of $k$ therefore satisfy the required estimate.  Consequently,
$\max_{0\leq k\leq n}|B_k|
 =O\!\left(\frac{{(m+1)}^n}{n^{3/2}}\right)$.

Theorem~\ref{thm:main} now follows from Lemma~\ref{lem:moving-sets}.

\section*{Declaration on the Origin of the Proof and Use of Generative AI}

The radius-one construction that motivated this manuscript was developed
through AI-assisted mathematical exploration with OpenAI Codex.  The named
authors independently reconstructed, reviewed, and verified every
mathematical argument and take full responsibility for all claims,
citations, and remaining errors.

\bibliographystyle{plain}
\bibliography{references}

\end{document}